\documentclass[a4paper,10pt]{amsart}
\usepackage{amscd}
\usepackage{hyperref}
\usepackage{graphicx}
\usepackage{ifpdf}
\ifpdf
\fi

\newcommand{\qbinom}[2]{\genfrac{[}{]}{0pt}{}{#1}{#2}}

\newcommand{\bZ}{\mathbb{Z}}

\newcommand{\bQ}{\mathbb{Q}}

\newcommand{\fg}{\mathfrak{g}}

\newcommand{\fsl}{\mathfrak{sl}}

\newcommand{\fso}{\mathfrak{so}}
\newcommand{\fsp}{\mathfrak{sp}}

\newcommand{\SL}{\mathrm{SL}}
\newcommand{\SO}{\mathrm{SO}}
\newcommand{\SU}{\mathrm{SU}}
\newcommand{\Sp}{\mathrm{Sp}}

\newcommand{\End}{\mathrm{End}}

\newtheorem{theorem}{Theorem}[section]
\newtheorem{defn}[theorem]{Definition}

\newtheorem{lemma}[theorem]{Lemma}

\newtheorem{cor}[theorem]{Corollary}

\newtheorem{prop}[theorem]{Proposition}

\newcommand{\incg}[2][.5in]{\setbox5=\hbox{\;\includegraphics[height=#1]{#2}\;}%
\dimen1=-#1\divide\dimen1 by 2\raise\dimen1\box5}
\newcommand{\incgs}[1]{\setbox5=\hbox{\;\includegraphics[height=.25in]{#1}\;}\raise-7pt\box5}

\begin{document}
\title{Recoupling theory for quantum spinors}
\author{Bruce W. Westbury}
\date{2009}

\begin{abstract}This paper extends the Birman-Wenzl category by including
a spin representation and then developing the recoupling theory,
following \cite[Chapter 11]{MR2418111}.
In particular this gives a $q$-analogue of the chromatic evaluation of a
spin network. The recoupling theory is developed up to an evaluation of
the $q$-analogues of the Fierz coefficients.
\end{abstract}
\maketitle
\tableofcontents

\section{Introduction}
At a technical level this paper starts with some simple skein relations and develops further
relations. The relations we start with are the relations \eqref{skein1}, \eqref{skein2}
which are essentially the Kauffman skein relations and then we also introduce the relations
\eqref{tad1}, \eqref{tad2}, \eqref{cliffo}, \eqref{cliffu}. Then we develop the consequences
of these relations in \S \ref{gc} culminating in the calculation of the $q$-analogue of the
Fierz coefficients. My intention has been to ensure that the paper can be read at this level
with no background beyond an understanding of skein relations. This is based on
\cite[Chapter 11]{MR2418111}. Although this book only appeared recently the content was worked
out around thirty years ago when it was ahead of its time.

However the paper is motivated and informed by several other subjects and it seems appropriate
to give some account of these subjects in a comparatively lengthy introduction. The reason for
this is that although each of these subjects is individually well-known I anticipate that there
is a comparitively small audience who will be familiar with all of these topics. These topics
are; creation and annihilation operators from quantum field theory, the representation theory
of the Drinfeld-Jimbo quantised enveloping algebras and spin network evaluations.

This is achieved by introducing the spin representation in this context and developing
the associated recoupling theory. This follows closely \cite[Chapter 11]{MR2418111}
and the other main aim of this paper is to give the $q$-analogue of the material
in this Chapter. This material is in turn based on \cite{MR626119}, \cite{MR677640}, \cite{MR669170}.

\subsection{Notation}
In this paper we will be using diagrams to represent tensors. In particular the labels on the
edges of diagrams are not indices but are used to denote different representations. 
Here we discuss the various
notations for the classical Clifford relations. Let $V$ be a vector space with a symmetric
inner product, $\left\langle-,-\right\rangle$. Then in the mathematics literature the Clifford algebra of $V$, $C(V)$, is defined
to be an algebra $C(V)$ with an inclusion of vector spaces $a\colon V\rightarrow C(V)$ which is
universal with the property that the following Clifford relations hold
\[ a(u)a(v)-a(v)a(u)=\left\langle u,v\right\rangle 1 \]
where $u,v\in V$. This means that a representation of $C(V)$ is a vector space $S$ and a linear
map $\rho\colon V\rightarrow\End(S)$ such that the following Clifford relations hold
\[ \rho(u)\rho(v)-\rho(v)\rho(u)=\left\langle u,v\right\rangle 1 \]
In the physics literature this is written in index notation. Instead of $\rho\colon V\rightarrow\End(S)$
we have $\gamma\colon V\otimes S\rightarrow S$ and this is written in index notation as
$\gamma^{\mu\, a}_b$. Here indices $\mu,\nu,\ldots$ are associated with $S$ and indices $a,b,c,\ldots$
are associated with $V$. Then the Clifford relations in index notation are
\[ \gamma^{\mu\, a}_b\gamma^{\nu\, b}_c+\gamma^{\nu\, a}_b\gamma^{\mu\, b}_c=g^{\mu\,\nu}\delta^a_c \]
where $g^{\mu\,\nu}$ is the metric tensor which gives the inner product. In this paper the vector
space $S$ is represented by a dashed line and the vector space $V$ by a solid line. The tensor
$\gamma$ is represented by a trivalent vertex
\[ \gamma^{\mu\, a}_b = \incg{spinors.156} \]
Then the Clifford relations are written as
\begin{equation*}
\incg{spinors.157}+\incg{spinors.158}=\incg{spinors.159}
\end{equation*}
However the indices convey no information and so are omitted and the Clifford relations
are written as
\begin{equation}\label{cliff}
\incg{spinors.1}+\incg{spinors.2}=\incg{spinors.3}
\end{equation}
This process can be reversed. The procedure for converting a diagram to index notation
consists in the following steps. Assume the diagram has been drawn in general position
with respect to the height function on the page.
Then take the set of vertices to be the trivalent vertices and the critical points.
Take the edges to be the lines connecting these vertices. Put an index $\mu,\nu,\ldots$
on each dashed line (with a different index for each line) and an index $a,b,\ldots$ on
each solid line (also all different). Then for each vertex write the tensor $\gamma$ or $g$
with the indices corresponding to the edges. Then all the internal edges have contracted
indices. In a linear combination of diagrams or in an equation the indices on the boundary
edges are required to be the same for all the diagrams.

Leaving aside aesthetic considerations and personal prejudices there are two reasons for
preferring the diagram notation over the index notation. The first reason is that there is
a choice involved in drawing the diagram in general position. There are (infinitely) many
index expressions which represent the same tensor and which are all represented by the
same diagram. One effect of this is that the proofs in index notation become longer as
each isotopy of diagrams has to be replaced by a sequence of tensors related by standard tensor
identities. The other reason is that the diagram notation is more general. The diagrams in this
paper can only be taken to be tensors if $z=q^a$ for some integer $a$.

\subsection{Coefficients}
On a first reading of this paper the ring of scalars should be taken to be $\bQ(q,z)$, the field
of rational functions in the indeterminates $q$ and $z$. However on subsequent readings the ring
of scalars should be taken to be an integral domain with this field as its field of fractions.
This is analogous to the situation with quantum groups. This are initially defined as Hopf algebras
over the field $\bQ(q)$ and then subsequently one can introduce the integral form which is defined
over $\bZ[q,q^{-1}]$, the ring of Laurent polynomials with integer coefficients.

The elements $[a]\in \bZ[q,q^{-1}]$ are defined for $a\in\bZ$ in the usual way by
\[ [a] = \frac{q^a-q^{-a}}{q-q^{-1}} \]
The $q$-binomial coefficients are also defined in the usual way and satisfy
\begin{equation}\label{qbinom}
\qbinom{a+b+1}{a}=q^a\qbinom{a+b}{a}+q^{-b-1}\qbinom{a+b}{a-1} 
\end{equation} 

The ring $R$ is constructed by taking the ring of Laurent polynomials $\bQ[q,q^{-1}]$
and adjoining elements $1/[a]$ for all $a>1$. 

Next we introduce a second indeterminate $z$ Instead of replacing $\bZ[q,q^{-1}]$ by
$\bZ[q,z,(zq)^{-1}]$ we take a blow-up. This is the approach taken in \cite{MR1312974}.
\begin{defn} The ring $K^\prime$ is the quotient of the
ring $\bZ[q,z,\delta,(qz)^{-1}]$ by the relation
\[ (q-q^{-1})\delta = (z-z^{-1}) \]
\end{defn}

Then we take the Kauffman skein relations with coefficients in $K^\prime$. The advantage
of this approach is that we have a homomorphisms $K^\prime \rightarrow \bZ[q,z,(zq)^{-1},(q-q^{-1})^{-1}]$
and $K^\prime \rightarrow \bZ[\delta]$ given by $q,z\mapsto 1$ and $\delta\mapsto \delta$.
Then under the first specialisation we recover the usual definition of the Birman-Wenzl algebras
and under the second specialisation we recover the Brauer algebras.

Next we extend the definition of the $q$-integers, $[a]$.
\begin{defn} For $a\in\bZ$ define $[n+a]\in K^\prime$ by
\begin{align*} [n+a] &=z[a] + q^{-a}\delta = z^{-1}[a]+q^a\delta \\
[2n+a] &=z[a] + q^{-a}(z+z^{-1})\delta = z^{-1}[a]+q^a(z+z^{-1})\delta 
\end{align*}
\end{defn}
It is straightforward to extend these definitions to define $[bn+a]$ for $b,a\in\bZ$.
However these do not appear in this paper. These satisfy the identity
\[ [A+B][C]=[A][B+C]+[B][A-C] \]
for all $A,B,C$ of the form $bn+a$ with $b,a\in\bZ$.

The ring of scalars in this paper is, $K$, a localisation of $K^\prime$. The elements that are
inverted are $[a]$ for all $a\in\bZ$ and $(zq^{-k}+z^{-1}q)$ for $k\ge 0$. The specialisations that
we are interested in arise from the following homomorphisms. For each $a\in bZ$ we have a
homomorphism $K\rightarrow R$ given by $q\mapsto q$, $z\mapsto q^a$, $\delta\mapsto [a]$.
There is also a homomorphism $K\rightarrow \bQ[\delta]$ given by $q,z\mapsto 1$, $\delta\mapsto\delta$.
Then for each integer $a$ we have a commutative diagram of ring homomorphisms
\[ \begin{CD}
 K @>{z\mapsto q^a}>> R \\
@VVV @VVV \\
\bQ[\delta] @>>{\delta\mapsto a}> \bQ
\end{CD} \]

\subsection{Quantum groups}
The approach in this paper is to avoid the theory of quantum groups. However it will be clear
to experts that this theory informs the results in this paper. Here we give some of the general
results from the theory of quantum groups that are implicit in this paper.

The first result gives the eigenvalues of the braid matrix acting on a tensor product.
This result is from \cite{MR1025154} and further details in the cases at hand are given
in \cite[(2.21) Proposition]{MR1427801}. This result is also stated in
\cite[Proposition 2.3.2]{MR1327533} and in \cite[Proposition 1.2]{MR1929189}.
The result is the following:
\begin{prop} Let $\lambda$, $\mu$, $\nu$ be dominant integral weights of a simple Lie
algebra $\fg$. Let $V(\lambda)$, $V(\mu)$, $V(\nu)$ be the associated irreducible highest
weight representations of $U_q(\fg)$. Assume $V(\nu)$ is a summand of $V(\lambda)\otimes V(\mu)$.
Put $c(\lambda)=(\lambda,\lambda+\rho)$ where $2\rho$ is the sum of the positive roots.
Let $R^{\lambda\, \mu}\colon V(\lambda)\otimes V(\mu)\rightarrow V(\mu)\otimes V(\lambda)$
be the braiding. Then the action of $R^{\lambda\, \mu}R^{\mu\, \lambda}$ on the irreducible components
$V(\nu)$ is scalar multiplication by $q^{c(\lambda)+c(\mu)-c(\nu)}$.
 \end{prop}

A further result is that the quantum dimension is given by the principal specialisation of
the Weyl character formula, see \cite[Proposition 10.10]{MR823672}. The result is the following:
\begin{prop} Let $\lambda$ be a dominant integral weights of a simple Lie
algebra $\fg$. Let $V(\lambda)$ be the associated irreducible highest
weight representation of $U_q(\fg)$. Then
\[ \dim_q(V(\lambda)=\prod_{\alpha\in \check\Delta_+}
 \frac{[\left\langle \lambda+\rho,\alpha\right\rangle ]}{[\left\langle \rho,\alpha\right\rangle ]}
\]
\end{prop}
The quantum dimensions of the representations of $\SL(n)$ are given in \cite{MR908150}.
The quantum dimensions of the vector representations of the classical groups are given
in \cite{MR1434112} and \cite{MR1090432} and the quantum dimensions of the spin
representations of the orthogonal groups are given in \cite{MR1929189}.

The $q$-analogue of the spin representation is given explicitly in 
\cite{MR1086739} and \cite[\S 8.4 \& 8.5]{MR1881971}.

\subsection{Yang-Baxter equation}
In this paper we consider four solutions of the Yang-Baxter equation and use these to
construct idempotents. Here we explain the construction. Assume that we have an algebra
$A(n)$ and that for $1\le i\le n-1$ we have $R_i(u)\in A(n)$ and these satisfy the equations
\begin{align} R_i(u)R_j(v)&=R_j(v)R_i(u)\qquad\text{for $|i-j|>1$}\label{R1}\\ R_i(u)R_{i+1}(uv)R_i(v)&=R_{i+1}(v)R_i(uv)R_{i+1}(u) \label{R2}\end{align}
The second equation is known as the Yang-Baxter equation.

Let $e_1,\ldots ,e_n$ be an orthonormal basis of a vector space $V$. Construct a root system
of type $A_{n-1}$ by taking the roots to be the set of vectors
$ \{ e_i-e_j | i\ne j \} $.
Take the positive roots to be 
$ \{ e_i-e_j | i< j \} $
and the simple roots to be 
$ \{ e_i-e_{i+1}  \} $.
The reflection group associated to this root system is the symmetric group.
The generators corresponding to the simple roots are the reflections $s_i$ where
$s_i$ acts on the basis by the transposition $(i,i+1)$. Take the group algebra of the free
abelian group on the basis to be the ring of Laurent polynomials in indeterminates
$q_1,\ldots ,q_k$.

Let $s_{i_1}\ldots s_{i_k}$ be a reduced word in the generators. Then associated to this
word is a sequence of positive roots $\alpha_{1}\ldots \alpha_{k}$. Hence we have an element
in $A(n-1)$ given by $R_{i_1}(q^{\alpha_1})\ldots R_{i_k}(q^{\alpha_k})$. Then the equations
\eqref{R1} and \eqref{R2} imply that this element only depends on the permutation represented
by the reduced word and not on the choice of reduced word.

This can be described using diagrams.
If we draw each generator as a crossing then any word in these generators can be drawn as a diagram.
The word is reduced if any two strings cross at most once. The positive root $\alpha$ associated
to a crossing is $e_j-e_i$.where string $i$ crosses string $j$ at the crossing and strings are
numbered by their starting point

In this paper we specialise to $q_j=q^j$ and we apply this construction to the permutation
$i\leftrightarrow n-i+1$. The properties of the idempotents that we construct follow from
the fact that this word is represented by any reduced word of length $n(n-1)/2$.

The four $R$-matrices we consider were first given in the seminal paper \cite{MR824090}.
The general method for finding solutions of the Yang-Baxter equation is to solve the
Jimbo equations for a finite dimensional representation of the quantised enveloping algebra
of an affine Kac-Moody algebra.

\subsection{Creation and annihilation operators}
A fundamental concept in quantum field theory are the creation and annihilation operators acting on Fock space. The Clifford algebras were introduced in \cite{MR1507084}. The $q$-analogue of the Clifford
algebras and the spin representations and of the Weyl algebra and the oscillator representation were
introduced in \cite{MR1015339} and are studied in \cite{MR1036118}, \cite{MR1327533}, \cite{MR1851791},
\cite{MR1673976}.

As a starting point let $V$ be the vector representation of $\SO(n,n)$. Then the Clifford algebra
of $V$, $C(V)$, is a filtered algebra and the associated graded algebra is $\Lambda^\bullet(V)$,
the exterior algebra of $V$. This is also the decomposition of $C(V)$ as a representation of
$\SO(n,n)$.
Alternatively one starts with a vector space $V$ with a symplectic form, $\omega$,
and defines the Weyl algebra $W(V)$ with the inclusion of vector spaces $a\colon V\rightarrow W(V)$ to be universal with the relation
\[ a(u)a(v)-a(v)a(u)=\omega( u,v) \]
The algebra $W(V)$ is filtered and the associated graded algebra is $S^\bullet(V)$, the symmetric
algebra of $V$. This is also the decomposition of $W(V)$ as a representation of $\Sp(V)$.

The Clifford algebra is simple. Let $S$ be an irreducible representation. The Clifford algebra
is also $\bZ_2$-graded, and so we have $C(V)=C^{\mathrm{even}}(V)\oplus C^{\mathrm{odd}}(V)$.
Then the restriction of $S$ to $C^{\mathrm{even}}(V)$ is the sum of two irreducible representations,
$S=S_+\oplus S_-$. The Lie algebra $\fso(n,n)$ is included in $C^{\mathrm{even}}(V)$ and this induces
a surjective homomorphism $U(\SO(n,n))\rightarrow C^{\mathrm{even}}(V)$. Similarly the Weyl algebra
is simple. Let $M$ be an irreducible representation. Then the Weyl algebra is also $\bZ_2$-graded, and so we have $W(V)=W^{\mathrm{even}}(V)\oplus W^{\mathrm{odd}}(V)$.
Then the restriction of $M$ to $W^{\mathrm{even}}(V)$ is the sum of two irreducible representations,
and these are irreducible representations of $\Sp(V)$.

Now assume $V=U\oplus U^*$. Then we have a natural symmetric inner product on $V$ and also a natural
symplectic form on $V$ each with with $U$ and $U^*$ isotopic. Then the relations for the Clifford algebra
are
\begin{align*}
 a(u)a(v)+a(v)a(u)&=0\\
a(\phi)a(\psi)+a(\psi)a(\phi)&=0\\
a(v)a(\phi)+a(\phi)a(v)&=\phi(v).1
\end{align*}
where $u,v\in U$ and $\phi,\psi\in U^*$.

This algebra acts on the vector space $S=\oplus_{p=0}^n\Lambda^p(U)$. The operators $a(u)$ are
called creation operators and the operators $a(\phi)$ are called annihilation operators.
Similarly, for the symplectic form, the algebra $W(V)$ acts on $M=\oplus_{p=0}^nS^p(U)$.

From the point of view of the representation theory this corresponds to inclusions
$\SL(n)\rightarrow \SO(n,n)$ and $\SL(n)\rightarrow \Sp(n,n)$.

Then for the quantum group version we replace $\fso(n,n)$ by the Drinfeld-Jimbo quantised enveloping
algebra $U_q(D_n)$ and we replace $\fsl(n)$ by $U_q(A_{n-1})$. Then the representations we are considering
are highest weight representations and so correspond to representations of the quantum group.
Then the Clifford algebra $C_q(n)$ is a filtered algebra and the decomposition into irreducible representations of $U_q(D_n)$ is $C_q(V)=\oplus_{k\ge 0} \Lambda^k (V)$. The Clifford algebra is also
$\bZ_2$-graded and we have a surjective algebra homomorphism $U_q(D_n)\rightarrow C_q^{\mathrm even}(n)$.

Alternatively we replace $\fsp(2n)$ by  $U_q(C_n)$ and we replace $\fsl(n)$ by $U_q(A_{n-1})$. Then the representations we are considering are lowest weight representations and so correspond to representations of the quantum group. In this case the lowest weights of $S_{\pm}$ are not integral and these representations have infinite dimension.
Then the Weyl algebra $W_q(n)$ is a filtered algebra and the decomposition into irreducible representations of $U_q(C_n)$ is $W_q(V)=\oplus_{k\ge 0} V(k\omega_1)$. The Weyl algebra is also
$\bZ_2$-graded and we have a surjective algebra homomorphism $U_q(C_n)\rightarrow W^{\mathrm{even}}_q(n)$.
For $n=1$ this is an isomorphism.

This account is based on the split real forms of the Lie algebras. In the physical applications
it is often the case that it is the compact real forms which are relevant. In this case the
relevant inclusions are $\SU(n)\rightarrow \SO(2n)$ and $\SU(n)\rightarrow \mathrm{H}(n)$; here $\mathrm{H}(n)$
is the quaternion unitary group and this is the compact real form of the complexification of
$\Sp(2n)$. This case requires more care as the endomorphism algebra of an irreducible real
representation can be the real numbers, the complex numbers or the quaternions whereas for the
split real form the endomorphism algebra is always the real numbers.

This account is based on the even dimensions. In odd dimensions intermediate between the groups
$\SO(n,n)$ we have the groups $\SO(n,n-1)$. Here there is one spin representation $S$.
For the inclusion $\SO(n,n-1)\rightarrow \SO(n,n)$ the restriction of both $S_+$ and $S_-$
is $S$. For the inclusion $\SO(n-1,n-1)\rightarrow \SO(n,n-1)$ both $S_+$ and $S_-$ restrict to $S$.
The situation is similar for the symplectic groups. Intermediate between the groups $\Sp(2n)$
are the odd symplectic groups $\Sp(2n-1)$. The representation $S$ is known as the oscillator
representation.

The point of view in this paper is that the Clifford algebra $C_q(V)$ can be constructed as an
algebra in the category of representations of $U_q(D_n)$ and the Weyl algebra $W_q$
can be constructed as an algebra in the category of representations of $U_q(C_n)$.
Then we regard $\mathrm{Mat}(\mathsf{B})$ as a category which interpolates between the categories
of representations of the algebras $U_q(D_n)$. Then we can construct an algebra $C_{q,z}$ in this
category which interpolated between the algebras $C_q$. From this point of view there is no difference
between the Clifford algebra and the Weyl algebra. The algebra $C_{q,z}$ is constructed by starting
with
\[ C_{q,z} = \oplus_{a\ge 0} V(a) \]
Then the multiplication is defined using \eqref{cp} and the associativity is given by the $6j$-symbols.
It is clear that this is a filtered and $\bZ_2$-graded algebra. It is also a Frobenius algebra.
From this point of view the aim of this paper is to compute trace maps for products in this algebra.

\subsection{Spin networks}
One of the motivations for this work was to find a $q$-analogue of the chromatic evaluation
of a spin network. In this paper we will change the terminology and refer to spin networks
simply as networks.

\begin{defn} Let $a,b,c\ge 0$. Then $(a,b,c)$ is admissible if $a+b+c$ is even and the following inequalities are satisfied
 \[ a+b\ge c\qquad b+c\ge a \qquad c+a\ge b \]
\end{defn}
The inequalities are known as the triangle inequalities since they are equivalent to the condition
that there is a Euclidean triangle with sides of lengths $(a,b,c)$.
\begin{lemma} The triple $(a,b,c)$ is admissible if and only if there exists $m,n,p\ge 0$
such that
\[ m+n=a\qquad n+p=b\qquad p+m=c \]
\end{lemma}

\begin{defn} A network is an isotopy class of trivalent graphs embedded in the sphere, $S^2$.
A labelled network is a labelling of the edges by non-negative integers such that for each 
vertex the three edge labels are an admissible triple.
\end{defn}

This definition allows edges labelled $0$. These edges can be omitted without any loss.

Associated to a labelled network is a labelled strand network.
A strand network is a 4-valent graph embedded in the plane together with
a rectangle at each vertex. The edges are called strands. A labelled strand
network is a labelling of the strands by non-negative integers such that
for each rectangle the sums of the two labels on the two opposite sides of the
rectangle are equal.

Given a network, the associated strand network is drawn by taking
the boundary of a thickening of the trivalent graph and then drawing a
solid rectangle for each edge of the trivalent graph. This can also be
considered as the medial graph, see \cite{MR1428870}. If the network is labelled
then the strand network is labelled by replacing each admissible triple $(a,b,c)$
by the corresponding $(m,n,p)$.

Given the labelled strand network $(N,L)$ the Penrose evaluation introduced in
\cite{MR0281657} is defined as follows.
Each strand labelled $a$ is replaced by $a$ parallel lines. A state
assigns a permutation to each rectangle. Let $S$ be a state, let
$|S|$ be the number of closed loops and let $\varepsilon(S)$ be $(-1)^C$
where $C$ is the total number of crossings in the state $S$. Then the evaluation is
\begin{equation}\label{evP} \chi_{-2}(N,L) = \sum_S \varepsilon(S)(-2)^{|S|} \end{equation}
where the sum is over all states $S$.

The chromatic evaluation is an extension of this definition. This can be defined by
replacing \eqref{evP} by
\begin{equation}\label{evCh} \chi_{\delta}(N,L) = \sum_S \varepsilon(S)\delta^{|S|} \end{equation}
where $\delta$ is an indeterminate. This evaluation is a polynomial in $\delta$.
The reason for introducing this evaluation is that if $\delta$ is taken to be a positive integer
then this has a combinatorial interpretation. The chromatic evaluation is studied from this
point of view in \cite[Chapter 8]{MR1280463} and \cite{MR1634471}. It is not practical to 
compute the chromatic evaluation from this definition. The purpose of \cite{MR1634471} is to
give a generating function for the evaluations of a fixed network. This is an effective method
for computing chromatic evaluations.

A different method for defining evaluations is to use tensors. The data here is a vector space
$V(a)$ for each non-negative integer $a$ with a non-degenerate inner product. Then for each
admissible triple $(a,b,c)$ we are given an element of $V(a)\otimes V(b)\otimes V(c)$.
Once this data is specified then we write the labelled network in index notation and convert
to a tensor using this data. This is a totally contracted tensor and is therefore a scalar.
This scalar is the evaluation. The tensors that are specified are required to satisfy some identities.
These are required so that the evaluation is well-defined and does not depend on the way the
labelled network is written as a totally contracted tensor.

The original Penrose evaluation is a tensor evaluation where the vector space $V(a)$ is the
irreducible representation of $\SL(2)$ with highest weight $a$ and dimension $a+1$.
This is also the $a$-th symmetric power of the two dimensional defining representation.
This was the motivation for the Penrose evaluation. Thus when the quantum group $U_q(\SL(2))$
was discovered it was natural to introduce the $q$-analogue of this evaluation. Accounts of
this evaluation are given in \cite[Chapters 1-9]{MR1280463},\cite{MR1366832} and \cite{MR1446615}.

A natural question at this point is to ask if there is a $q$-analogue of the chromatic evaluation.
The key to this is the observation that the chromatic evaluation for $n$ a positive integer
is also a tensor evaluation. I realised this somewhat belatedly on reading \cite[Chapter 11]{MR2418111}.
Here the vector space $V(a)$ is the $a$-th exterior power of the vector representation
of $\SO(k,k)$ for $n=2k$ or of $\SO(k+1,k)$ for $n=2k+1$. This gives $q$-analogues of these
evaluations by replacing $\SO(k,k)$ by the quantised enveloping algebra $U_q(D_k)$ and
by replacing $\SO(k+1,k)$ by the quantised enveloping algebra $U_q(B_k)$. Denote these evaluations
by $\chi^{(q)}_n$. Then a $q$-analogue of the chromatic evaluation is an evaluation with values
in $K$ such that applying the homomorphism $K\rightarrow\bQ[\delta]$ gives the chromatic evaluation
and such that for each positive integer $n$ applying the hommorphism $K\rightarrow R$, $z\mapsto q^n$
gives the evaluation $\chi^{(q)}_n$. An evaluation with these properties is defined in \S \ref{orth}.

\section{Special linear groups}
\subsection{Skein relations}
\begin{defn} The category $\mathsf{H}$ is constructed by taking the free $K$-linear category on the category
of oriented framed tangles and then imposing the following skein relations:
\begin{equation}\label{hecke2}
\incg{spinors.136}=\delta\qquad
\incg{spinors.138}=z\incg{spinors.139}\qquad 
\incg{spinors.137}=z^{-1}\incg{spinors.139}
\end{equation}
\begin{equation}\label{hecke}
 \incg{spinors.43}-\incg{spinors.42}=(q-q^{-1})\incg{spinors.41}
\end{equation} 
This relation can also be written as $\sigma-\sigma^{-1}=q-q^{-1}$.
\end{defn}

For $\SL(n)$ the braid matrices are given in \cite[(17)]{MR629943} and \cite[(3.5)]{MR824090}.
The braid matrices are given in \cite[(3.6)]{MR824090} and \cite[\S 5]{MR1090432}.

\subsection{Yang-Baxter equation}
Next we introduce a solution of the Yang-Baxter equation.
\begin{defn}\label{Ru} Define $R_i(u)$ by
\begin{equation*}
(uq-u^{-1}q^{-1})R_i(u) = u\sigma_i - u^{-1}\sigma_i^{-1}
\end{equation*}
\end{defn}
\begin{prop}
This satisfies the Yang-Baxter equation \eqref{R2}.
It also satisfies unitarity
\[ R_i(u)R_i(u^{-1}) = 1 \qquad R_i(1) = 1 \]
\end{prop}
\begin{proof} This can be checked by a direct calculation. Alternatively,
it is sufficient to check this for each irreducible representation of the
three string Hecke algebra. The dimensions of the irreducible representations are
1, 2, 1. The relation is clear in the one dimensional representations. The two dimensional
representation is given by
\[ \sigma_1^{\pm 1}\mapsto
\begin{pmatrix} q^{\pm 1} & 0 \\ 1 & -q^{\mp 1}\end{pmatrix}
\qquad\sigma_2^{\pm 1}\mapsto
\begin{pmatrix} -q^{\mp 1} & 1 \\ 0 & q^{\pm 1}\end{pmatrix} \]
Hence it is sufficient to check that the following matrices satisfy the Yang-Baxter equation.
\begin{align*} R_1(u) &= \begin{pmatrix}
    uq-u^{-1}q^{-1} & 0 \\ u-u^{-1} & u^{-1}q-uq^{-1}
   \end{pmatrix}\\
R_2(u) &= \begin{pmatrix}
    u^{-1}q-uq^{-1} & u-u^{-1} \\ 0 & uq-u^{-1}q^{-1}
   \end{pmatrix}
 \end{align*}
\end{proof}

Then we define the sequence $F(p)$ by $F(1)=1$ and
\[ F(p+1)=F(p)R_{p}(q^{p})F(p) \]
Then these are idempotents.
The element $F(p)$ is characterised up to a scalar factor by the properties
\begin{equation}\label{idRh}
\sigma_iF(p)=qF(p)=F(p)\sigma_i
\end{equation}
for $1\le i\le p-1$. The element $F(p)$ is characterised uniquely by these properties
together with the property that it is idempotent.

Then we have the recurrence relation
\[ [p+1]\dim_q F(p+1) = [n+p] \dim_q F(p) \]
and the initial conditions $\dim_q F(0)=1$, $\dim_q F(1)=[n]$.
The solution to this recurrence relation is
\[ \dim_q F(p) = \qbinom{n+p-1}{p} \]

Next we introduce a solution of the Yang-Baxter equation.
\begin{defn}\label{Rt} Define $S_i(u)$ by
\begin{equation*}
(uq-u^{-1}q^{-1})S_i(u) = u^{-1}\sigma_i - u\sigma_i^{-1}
\end{equation*}
\end{defn}
\begin{prop}
This satisfies the Yang-Baxter equation \eqref{R2}.
It also satisfies unitarity
\[ S_i(u)S_i(u^{-1}) = 1 \qquad S_i(1) = 1 \]
\end{prop}
\begin{proof} This can be checked by the same methods as for the $R$-matrix in Definition
\ref{Ru}. Alternatively it follows by noting that there is an involution which interchanges
these two $R$-matrices.\end{proof}
Then we define the sequence $E(p)$ by $E(1)=1$ and
\[ E(p+1)=E(p)S_{p}(q^{p})E(p) \]
Then these are idempotents.
The element $E(p)$ is characterised up to a scalar factor by the properties
\begin{equation}\label{idSh}
\sigma_iE(p)=-q^{-1}E(p)=E(p)\sigma_i
\end{equation}
for $1\le i\le p-1$. The element $E(p)$ is characterised uniquely by these properties
together with the property that it is idempotent.

Then we have the recurrence relation
\[ [p+1]\dim_q E(p+1) = [n-p] \dim_q E(p) \]
and the initial conditions $\dim_q E(0)=1$, $\dim_q E(1)=[n]$.
The solution to this recurrence relation is
\[ \dim_q E(p) = \qbinom{n}{p} \]

Using this we define the exterior powers and the two types of trivalent vertices.
These satisfy the relations in \cite[Lemma A.1]{MR1659228}.

\subsection{Clifford relations}
Then we have the following relations which are consistent with \eqref{hecke}.
\begin{align*}
 q^{-1}\incg{spinors.44}+\incg{spinors.46}&=0\\
 q\incg{spinors.44}+\incg{spinors.45}&=0 \\
 q^{-1}\incg{spinors.47}+\incg{spinors.48}&=0\\
 q\incg{spinors.47}+\incg{spinors.49}&=0 
\end{align*}
The following relations which are consistent with \eqref{hecke} and with taking traces.
\begin{align*}
 \incg{spinors.50}+\incg{spinors.56}&=z^{-1}q^{a}\incg{spinors.53}\\
 \incg{spinors.50}+\incg{spinors.55}&=zq^{-a}\incg{spinors.53}\\
 \incg{spinors.54}+\incg{spinors.52}&=q^{-a}\incg{spinors.57}\\
 \incg{spinors.54}+\incg{spinors.51}&=q^{a}\incg{spinors.57}\\
\end{align*}
These are the $q$-analogue of the canonical anticommutation relations.

For the first equation the upper trace condition is equivalent to the identity
\[ [a]+z^{-1}[n-a]=z^{-1}q^a[n] \]
For the first equation the lower trace condition is equivalent to the binomial identity
\[ \qbinom{n-1}{a-1}+z^{\pm 1}\qbinom{n-1}{a} = z^{-1}q^a\qbinom{n}{a} \]

\section{Orthogonal groups}\label{orth}
\subsection{Skein relations}
\begin{defn}
The category $\mathsf{B}$ is constructed by taking the free $K$-linear category on the category
of unoriented framed tangles and then imposing the following skein relations:
\begin{equation}\label{skein1}
\incg{spinors.10}=(zq^{-1}+z^{-1}q)\delta\qquad
\incg{spinors.12}=z^{-2}q\incg{spinors.13}\qquad 
\incg{spinors.11}=z^2q^{-1}\incg{spinors.13}
\end{equation}

\begin{equation}\label{skein2}
\incg{spinors.7}-\incg{spinors.6}=(q-q^{-1})
\left(\incg{spinors.8}-\incg{spinors.9}\right)
\end{equation}
These equation are also written as
\[ \sigma_iu_i=z^{-2}qu_i\qquad \sigma_i^{-1}u_i= z^{2}q^{-1}u_i \]
\[ \sigma_i - \sigma_i^{-1} = (q-q^{-1})(1-u_i) \]
\end{defn}

For $\SO(n,n)$ we have
\begin{equation}
 u = \sum_{i,j}q^{(i+j)/2} E_{ij}\otimes E_{-i-j}
\end{equation}
and the braid matrices are
\begin{align}
 \sigma &= \sum_i\left( qE_{ii}\otimes E_{ii} + q^{-1}E_{i-i}\otimes E_{-ii}\right)
+\sum_{i\ne j,-j} E_{ij}\otimes E_{ji} \\
&\qquad +(q-q^{-1})\sum_{i<j} E_{ii}\otimes E_{jj} 
 -(q-q^{-1})\sum_{j<-i} q^{(i+j)/2}E_{ij}\otimes E_{-i-j} \nonumber
\end{align}
\begin{align}
 \sigma^{-1} &= \sum_i\left( q^{-1}E_{ii}\otimes E_{ii} - qE_{-ii}\otimes E_{i-i}\right) 
 +\sum_{i\ne j,-j} E_{ij}\otimes E_{ji} \\
&\qquad -(q-q^{-1})\sum_{i>j} E_{ii}\otimes E_{jj} 
 +(q-q^{-1})\sum_{j>-i} q^{(i+j)/2}E_{ij}\otimes E_{-i-j} \nonumber
\end{align}
Then these satisfy
\[ \sigma - \sigma^{-1} = (q-q^{-1})(1-u) \]

\subsection{Yang-Baxter equation}
Next we introduce a solution of the Yang-Baxter equation. These solutions are associated with the quantised enveloping algebras of the Kac-Moody algebras $D_n^{(1)}$.
\begin{defn}\label{RmatrixD} Define $R_i(u)$ by
\begin{align*}
 (uzq^{-1}-u^{-1}&z^{-1}q)(uq-u^{-1}q^{-1})R_i(u) = \\
&(u-u^{-1})(uzq^{-1}\sigma_i-u^{-1}z^{-1}q\sigma_i^{-1}) +(zq^{-1}-z^{-1}q)(q-q^{-1}) \\
=&(uzq^{-1}-u^{-1}z^{-1}q)(u\sigma_i-u^{-1}\sigma_i^{-1})+(zq^{-1}-z^{-1}q)(q-q^{-1})u_i
\end{align*}
\end{defn}
\begin{prop}
This satisfies the Yang-Baxter equation \eqref{R2}.
It also satisfies unitarity
\[ R_i(u)R_i(u^{-1}) = 1 \qquad R_i(1) = 1 \]
\end{prop}
\begin{proof} This can be checked by a direct calculation. Alternatively, it is sufficient
to check this for each irreducible representation of the three string Birman-Wenzl algebra.
If we impose the relation $u_i=0$ then we obtain the $R$-matrix in Definition \ref{Ru}.
Hence the Yang-Baxter equation is satisfied in all the representations with $u_i=0$.
This leaves the following three dimensional representation
\[ \sigma_1^{\pm 1}\mapsto\begin{pmatrix}
z^{\mp 2}q^{\pm 1} & 0 & 0 \\ -z^{\pm 1}(zq^{-2}+z^{-1}q^2) & -q^{\mp 1} & 0 \\
q^{\mp 1} & 1 & q^{\pm 1}
\end{pmatrix} \]
\[ \sigma_2^{\pm 1}\mapsto\begin{pmatrix}
q^{\pm 1} & 1 & q^{\mp 1} \\ 0 & -q^{\mp 1} & -z^{\pm 1}(zq^{-2}+z^{-1}q^2)\\
0 & 0 & z^{\mp 2}q^{\pm 1}
\end{pmatrix} \]
Hence it is sufficient to check that the Yang-Baxter equation holds in this representation.
\end{proof}

It also has crossing symmetry. If we rotate diagrams through a quarter of a revolution
we have $1 \leftrightarrow u_i$ and $\sigma_i\leftrightarrow\sigma_i^{-1}$. Then crossing 
symmetry says that we also have
\begin{multline}
 (uzq^{-1}-u^{-1}z^{-1}q)(uq-u^{-1}q^{-1})R_i(u)\leftrightarrow \\
 (u-u^{-1})(uzq^{-2}-u^{-1}z^{-1}q^{-2} )R_i(u^{-1}z^{-1}q)
\end{multline}
Furthermore if we take the quotient $u_i=0$ we get the $R$-matrix in Definition \ref{Ru}.

The reason we have introduced this is that we can define a sequence of idempotents by
\begin{equation}
 F(1)=1 \quad F(p+1)=F(p)R_{p}(q^{p})F(p)
\end{equation}
These idempotents project onto the representation $V(p\omega_1)$.
In this paper we will not make use of these idempotents.
The element $F(p)$ is characterised up to a scalar factor by the properties
\begin{equation}\label{idR}
u_iF(p)=0=F(p)u_i\qquad
\sigma_iF(p)=qF(p)=F(p)\sigma_i
\end{equation}
for $1\le i\le p-1$. The element $F(p)$ is characterised uniquely by these properties
together with the property that it is idempotent.

\begin{lemma}\label{lemR} For $p\ge 2$ and $1\le i\le p-1$
 \begin{equation*}
  R_i(u)F(p)=F(p)=F(p)R_i(u)
 \end{equation*}
\end{lemma}
\begin{proof}
 This is a direct calculation from Definition \ref{RmatrixD}
using \eqref{idR}.
\end{proof}
These idempotents are also given in \cite{MR1751618}.

As an excercise we show that the quantum trace of the idempotent $F(p)$ gives the quantum
dimensions of the representations $V(p\omega_1)$. This gives that for $p\ge 1$,
\begin{align}
 [n+p-1][p+1]\dim_q F(p+1) &= \left(
[n+p-1][2n+p-1]+[n-1]
\right) \dim_q F(p) \\
&=[2n+p-2][n+p]\dim_q F(p)
\end{align}
The initial condition is
\[ \dim_q F(1) = \frac{[2n-2]}{[n-1]}[n] \]

This quantum dimension can also be calculated as the principal specialisation of the Weyl character formula. This gives
\begin{equation}
 \dim_q V(p\omega_1) = \frac{[n+p-1]}{[n-1]}\qbinom{2n+p-3}{p}
\end{equation}
Hence
\begin{equation}
 \dim_q V((p+1)\omega_1) = \frac{[2n+p][n+p]}{[n+p-1][p+1]}
 \dim_q V(p\omega_1)
\end{equation}

Next we introduce a further solution of the Yang-Baxter equation. These solutions are associated with the quantised enveloping algebras of the Kac-Moody algebras $A_{2n}^{(2)}$.
\begin{defn}\label{RmatrixA} Define $S_i(u)$ by
\begin{align*}
 (uz^{-1}+u^{-1}&z)(uq-u^{-1}q^{-1})S_i(u)  \\
=&(u^{-1}-u)(u^{-1}z\sigma_i+uz^{-1}\sigma_i^{-1}) +(z+z^{-1})(q-q^{-1}) \\
=&(uz^{-1}+u^{-1}z)(u^{-1}\sigma_i-u\sigma_i^{-1})+(z+z^{-1})(q-q^{-1})u_i
\end{align*}
\end{defn}
\begin{prop}
This satisfies the Yang-Baxter equation \eqref{R2}.
It also satisfies unitarity
\[ S_i(u)S_i(u^{-1}) = 1 \qquad S_i(1) = 1 \]
\end{prop}
\begin{proof} This is a direct calculation.\end{proof}

It also has crossing symmetry. If we rotate diagrams through a quarter of a revolution
we have $1 \leftrightarrow u_i$ and $\sigma_i\leftrightarrow\sigma_i^{-1}$. Then crossing 
symmetry says that we also have
\begin{equation}
 (uz^{-1}+u^{-1}z)(uq-u^{-1}q^{-1})S_i(u)\leftrightarrow 
 (u-u^{-1})(uz^{-1}q^{-1}+u^{-1}zq )S_i(\imath u^{-1}z)
\end{equation}
Furthermore if we take the quotient $u_i=0$ we get the $R$-matrix in Definition \ref{Rt}.

The reason we have introduced this is that we can define a sequence of idempotents by
\begin{equation}
 E(1)=1 \quad E(p+1)=E(p)S_{p}(q^{p})E(p)
\end{equation}
These idempotents project onto the exterior powers of the vector representation.
The element $E(p)$ is characterised up to a scalar factor by the properties
\begin{equation}\label{idS}
u_iE(p)=0=E(p)u_i\qquad
\sigma_iE(p)=-q^{-1}E(p)=E(p)\sigma_i
\end{equation}
for $1\le i\le p-1$. The element $E(p)$ is characterised uniquely by these properties
together with the property that it is idempotent.

\begin{lemma}\label{lemS} For $p\ge 2$ and $1\le i\le p-1$
 \begin{equation*}
  S_i(u)E(p)=E(p)=E(p)S_i(u)
 \end{equation*}
\end{lemma}
\begin{proof}
 This is a direct calculation from Definition \ref{RmatrixA}
using \eqref{idR}.
\end{proof}
These idempotents are also given in \cite{MR1751618}.

As an excercise we show that the quantum trace of the idempotent $E(p)$ gives the quantum
dimensions of the exterior powers of the vector representation, $V(\omega_p)$. This gives
that for $p\ge 1$,
\begin{align*}
 \frac{[2n-2p]}{[n-p]}&[p+1]\dim_q E(p+1)\\
 &= \left(
\frac{[2n-2p]}{[n-p]}[2n-p-1]+\frac{[2n]}{[n]}
\right) \dim_q E(p) \\
&= \frac{[2n-2p-2]}{[n-p-1]}[2n-p]\dim_q E(p)
\end{align*}
The initial condition is
\[ \dim_q E(1) = \frac{[2n-2]}{[n-1]}[n] \]

This quantum dimension can also be calculated as the principal specialisation of the Weyl character formula. This gives
\begin{align}
 \dim_q V(\omega_p) &= \frac{[2n-2p]}{[n-p]}\frac{[n]}{[1]}
\prod_{k=1}^{p-1}\frac{[2n-k]}{[k+1]} \\
&= \left( \frac{[2n-2][n]}{[n-1][2n]}\right) \qbinom{2n}{p}
\end{align}
Hence
\begin{equation}
 \dim_q V(\omega_{p+1}) = \frac{[2n-2p-2][2n-p][n-p]}{[2n-2p][n-p-1][p+1]}
 \dim_q V(\omega_p)
\end{equation}

\begin{prop}\label{dimq} For $a\ge 0$.
 \begin{equation*}
  \incg[1in]{spinors.150}=\left( \frac{[2n-2][n]}{[n-1][2n]}\right) \qbinom{2n}{a}
 \end{equation*}
\end{prop}

The quantum dimensions for partitions are derived from the Weyl character formula in
\cite[Theorem 5.5]{MR1090432} and \cite{MR1434112}.

\subsection{Braiding}
First we deal with loops following \cite[\S 5.2 Proposition 6]{MR1280463}.
\begin{lemma}
 \begin{equation*}
 \incg{spinors.104}=z^{2a}q^{-a^2}\incg{spinors.105}
\end{equation*}
\end{lemma}

\begin{proof}
 The proof is by induction on $a$. The basis of the induction is the case $a=1$
which is given in \eqref{skein1}. The inductive step is the following calculation.

\begin{equation*}
 \incg{spinors.106}=\incg{spinors.107}
=(z^{2a}q^{-a^2})(z^{2b}q^{-b^2})q^{-2ab}\incg{spinors.108}
\end{equation*}
\end{proof}

\begin{lemma}\label{twist}
 \begin{equation*}
 \incg{spinors.109}=(-1)^{st+rs+rt}z^{-2s}q^{(s+t)(s+r)-2rt}\incg{spinors.110}
\end{equation*}
\end{lemma}
\begin{proof} This is a straightforward calculation.
\end{proof}

\subsection{Restriction}
There is an inclusion of $\SL(n)$ in $\SO(n,n)$. This extends to the quantised enveloping algebras
and so gives a restriction functor from the representation theory of $SO(n,n)$ to the representation
theory of $\SL(n)$.

Next we want to construct the restriction functor. For this we need the category
$\mathrm{Mat}(\mathsf{H})$ where $\mathsf{H}$ is the Hecke category. This construction
is discussed in \cite[Chapter 2]{scottthesis}. The objects of $\mathrm{Mat}(\mathsf{H})$
are ordered lists of objects of $\mathsf{H}$ and morphisms are matrices whose entries
are morphisms in $\mathsf{H}$.

Then we start with
\begin{equation}
 \incg{spinors.58}\mapsto \incg{spinors.59}\oplus \incg{spinors.60}
\end{equation}
This extends to
\begin{equation}
 \incg{spinors.80}\mapsto \bigoplus_{r+s=p}\incg{spinors.81}
\end{equation}
Then we define a linear order for the tensor product by
\begin{equation}
 \incgs{spinors.8}\mapsto \incgs{spinors.61}\oplus \incgs{spinors.62}
\oplus \incgs{spinors.63}\oplus \incgs{spinors.64}
\end{equation}

Then the Birman-Wenzl category $\mathsf{B}$ is finitely generated as a monoidal category.
Thus a monoidal functor $\mathsf{B}\rightarrow\mathrm{Mat}(\mathsf{H})$ is uniquely determined
by the values on the generators. The functor we are interested in is determined by the following.
First we have the braid group generators and their inverses
\begin{equation}
 \incg{spinors.7}\mapsto\begin{pmatrix}
\incgs{spinors.65} & 0 & 0 & 0 \\
0 & (q-q^{-1})\left(\incgs{spinors.62}-z^{-1}q\incgs{spinors.83}\right) & \incgs{spinors.66} & 0 \\
0 & \incgs{spinors.67} & 0 & 0 \\
0 & 0 & 0 & \incgs{spinors.68}
\end{pmatrix}\end{equation}
\begin{equation}
 \incg{spinors.6}\mapsto\begin{pmatrix}
\incgs{spinors.160} & 0 & 0 & 0 \\
0 & 0 & \incgs{spinors.161} & 0 \\
0 & \incgs{spinors.162} & (q-q^{-1})\left(-\incgs{spinors.63}+zq^{-1}\incgs{spinors.82}\right) & 0 \\
0 & 0 & 0 & \incgs{spinors.163}
\end{pmatrix}\end{equation}
Each of these is obtained from the other by switching crossings, reversing directions
and applying the involution $q\leftrightarrow q^{-1}$, $z\leftrightarrow z^{-1}$.

Then we have the maximum and minimum.
\begin{equation}
 \incgs{spinors.69}\mapsto\begin{pmatrix}
0 & z^{-1}q\incgs{spinors.71} & \incgs{spinors.70} & 0
\end{pmatrix}\end{equation}

\begin{equation}
 \incgs{spinors.72}\mapsto\begin{pmatrix}
0 & \incgs{spinors.73} & zq^{-1}\incgs{spinors.74} & 0
\end{pmatrix}^{\mathrm{T}}\end{equation}

These imply that
\begin{equation}
 \incg{spinors.9}\mapsto\begin{pmatrix}
0 & 0 & 0 & 0 \\
0 & z^{-1}q\incgs{spinors.83} & \incgs{spinors.84} & 0 \\
0 & \incgs{spinors.85} & zq^{-1}\incgs{spinors.82} & 0 \\
0 & 0 & 0 & 0
\end{pmatrix}\end{equation}

Then it is clear that the tangle relations are satisfied so this gives a spherical braided functor
$\mathsf{T}\rightarrow\mathrm{Mat}(\mathsf{H})$. Then it can be checked that the skein relations
\eqref{skein1} and \eqref{skein2} are satisfied so this gives a $K$-linear spherical braided functor
$\mathsf{B}\rightarrow\mathrm{Mat}(\mathsf{H})$.

\section{Spin representation}\label{gc}
In this section we introduce the $q$-analogue of the Clifford relations and develop some consequences.
The spin representation, $S$, is denoted by a new type of edge which in this paper is a dashed line.
Then we also introduce a new trivalent vertex corresponding to the linear map $V\otimes S\rightarrow S$
which is also known as Clifford multiplication and in the physics literature is refered to as a $\gamma$-matrix. Later we will also introduce trivalent vertices for the linear maps
$\Lambda^k(V)\otimes S\rightarrow S$ which are known in the physics literature as $\Gamma$-matrices.
We will refer to networks with these trivalent vertices as spinor networks.

The evaluation of a spinor network is defined either by using the restriction functor or
alternatively by using Proposition \ref{exp}. Neither of these definitions is practical.


Next we introduction the notation $\{k\}$ defined by
\[ \{k\} =zq^{-k}+z^{-1}q^k=\frac{[2n-2k]}{[n-k]} \]

\subsection{Clifford relations}
Then introducing the spin representation we assume the relations:
\begin{equation}\label{tad1}
 \incg{spinors.14}=\Delta\qquad\incg{spinors.19}=0
\end{equation}
\begin{equation}\label{tad2}
 \incg{spinors.15}=\delta\incg{spinors.16}\qquad
 \incg{spinors.17}=\frac{\Delta}{(zq^{-1}+z^{-1}q)}\incg{spinors.18}
\end{equation}
These relations hold for both orientations of the spinor line.

The $q$-analogue of the Clifford relation \eqref{cliff} are the two relations
\begin{equation}\label{cliffo}
q^{-1}\incg{spinors.1}+\incg{spinors.5}=z^{-1}
\incg{spinors.3}
\end{equation}

\begin{equation}\label{cliffu}
q\incg{spinors.1}+\incg{spinors.4}=z\incg{spinors.3}
\end{equation}

Note that 
\begin{align*}
 (q+q^{-1})F_i(2) &= q^{-1}+\sigma_i-\frac{1}{z\delta}u_i\\
&= q+\sigma_i^{-1}-\frac{z}{\delta}u_i
\end{align*}
so that these two relations are equivalent to
\begin{equation*}
 \incg{spinors.135}=0
\end{equation*}

Now to extend the restriction functor to the spin representation we start with
\begin{equation}
 \incg{spinors.75}\mapsto \bigoplus_{p\ge 0} \incg{spinors.76}
\end{equation}

Then we extend this by
\begin{equation}
 \incg{spinors.77}\mapsto \bigoplus_{p\ge 0}\left( q^{p}\incg{spinors.78}\oplus \incg{spinors.79}\right)
\end{equation}
\begin{equation}
 \incg{spinors.86}\mapsto \bigoplus_{p\ge 0}\left( q^{-p}\incg{spinors.87}\oplus \incg{spinors.88}\right)
\end{equation}
These coefficients are determined by the first equation in \eqref{tad2}.
Then we check that the canonical anticommutation relations imply the Clifford relations
\eqref{cliffo} and \eqref{cliffu}.

This gives an evaluation as we can restrict and then evaluate using \cite{MR1659228}.
\subsection{Expansion}
The following is a consequence of the Clifford relations
\begin{lemma} For all $p\ge 1$
\begin{equation}
 \incg{spinors.20}=\sum_{k=0}^{p-1}(-q)^kz
\incg{spinors.21}+(-q)^p\incg{spinors.22}
\end{equation}
\end{lemma}

\begin{proof} The proof is by induction on $p$. The basis of the induction is the 
case $p=1$ which is the Clifford relation \eqref{cliffu}.
\end{proof}

Next we take the trace of this identity.
\begin{lemma} For all $p\ge 1$
\begin{equation}
 \incg{spinors.23}=z^{2}q^{-1}\incg{spinors.24}
\end{equation}
\end{lemma}
\begin{proof}
 This follows from the following isotopy together with \eqref{skein1}.
\[ \incg{spinors.23}=\incg{spinors.151}=\incg{spinors.152}\]
\end{proof}

Then it follows easily from these two lemmas that:
\begin{prop}\label{exp} For all $p\ge 0$
\begin{equation}\label{expq}
 \incg{spinors.24}=\left(\frac{z}{z^{2}q^{-1}-(-q)^{p+1}}\right)
\sum_{k=0}^{p}(-q)^k \incg{spinors.25}
\end{equation}
\end{prop}
Note that the case $p=0$ is the third relation in \eqref{tad2}.

The significance of this Proposition is that it defines the evaluation
of a spinor network.

\begin{prop} Every spinor network has an evaluation.
\end{prop}
\begin{proof}First expand each idempotent so the evaluation is written as a linear
combination of spinor networks in which all edges are labelled one. Then evaluate this
by induction on the number of edges. The inductive step is an application of Proposition \ref{exp}.

This is well-defined since for each $n$ it is well-defined under the specialisation $z\mapsto q^n$.
\end{proof}

An immediate corollary is
\begin{cor} For all $a\ge 1$
 \begin{equation}
 \incg{spinors.27}=0
\end{equation}
\end{cor}

A further corollary is
\begin{cor} For all $p\ge 1$ such that $p$ is odd
\begin{equation}
 \incg{spinors.26}=0
\end{equation}
\end{cor}
\begin{proof} The proof is by induction using \eqref{tad1} as the basis of the induction
 and Proposition \ref{exp} for the inductive step.
\end{proof}

The usual way of stating this result is to say that the trace of a product of an odd number of
$\gamma$-matrices is zero. In view of this we are only interested in Proposition in the case 
that $p$ is even. In this case we can write \eqref{expq} as
\begin{equation}\label{expz}
 \incg{spinors.24}=\frac{q^{-p/2}}{zq^{-p/2-1}+z^{-1}q^{p/2+1}}
\sum_{k=0}^{p}(-q)^k \incg{spinors.25}
\end{equation}

\subsection{First evaluations}
Here we give a $q$-analogue of \cite[\S 11.1 (11.5)]{MR2418111}.
This is also the case $a=1$ of \cite[\S 11.2 (11.19)]{MR2418111}.

\begin{lemma}\label{lem1}
 \begin{multline*}
  \incg{spinors.89}=
\incg{spinors.1}- \\
\frac{(u-u^{-1})}{(uq-u^{-1}q^{-1})}\frac{(u+u^{-1})}{(uz^{-1}+u^{-1}z)}\incg{spinors.3}
 \end{multline*}
\end{lemma}
\begin{proof} This is a direct calculation from Definition \ref{RmatrixA}
 using the Clifford relations \eqref{cliffo} and \eqref{cliffu}.
\end{proof}

\begin{prop}\label{gamma} For all $p\ge 0$,
\begin{equation*}
 \incg{spinors.90}=\incg{spinors.91}-\frac{[n-p-1]}{[2n-2p-2]}[p+1]\incg{spinors.92}
\end{equation*}
\end{prop}
\begin{proof} The proof is by induction on $p$. The case $p=0$ is Lemma \ref{lem1} with $u=q$.
The inductive step is the following calculation. In the diagrams a crossing represents $S(q^{p+2})$.
\begin{align*}
\incg{spinors.93} &= \incg{spinors.94} \\
&=\incg{spinors.95}-\frac{[n-p-1]}{[2n-2p-2]}[p+1]\incg{spinors.96}\\
&=\incg{spinors.97}-\frac{[2p+4]}{[p+3]}\frac{[n-p-2]}{[2n-2p-4]}\incg{spinors.98}\\
&\qquad -\frac{[n-p-2]}{[2n-2p-4]}\frac{[p+1][p+2]}{[p+3]}\incg{spinors.98}\\
&=\incg{spinors.97}-\frac{[n-p-2]}{[2n-2p-4]}[p+2]\incg{spinors.98}
\end{align*}

The first step is Definition \ref{RmatrixA}. The second step is the inductive hypothesis.
For the third step we simplify the first diagram using Lemma \ref{lem1} and we simplify the second diagram using crossing symmetry and Lemma \ref{lemS}. The fourth step is the identity
\[ [2p+4]+[p+1][p+2]=[p+2][p+3] \]
\end{proof}

\begin{prop}\label{pid} For all $a\ge 1$
\begin{equation}
 \incg{spinors.28} = \Delta \prod_{k=1}^a \frac{[k]}{\{k\}} \incg{spinors.29}
\end{equation}
\end{prop}
\begin{proof} The proof is by induction on $a$. The basis of the induction is
$a=1$ which is a relation in \eqref{tad2}. The inductive
step is the following calculation
\begin{align*}
 \incg{spinors.30}&=\frac{q^{-a+1}}{zq^{-a}+z^{-1}q^{a}}
 \sum_{k=0}^{a-1}(-q)^k\incg{spinors.31} \\
&=\frac{q^{-a+1}}{zq^{-a}+z^{-1}q^{a}}
 \sum_{k=0}^{a-1}(-q)^{2k}\incg{spinors.32} \\
&=\frac{q^{-a+1}}{zq^{-a}+z^{-1}q^{a}}
  \left(\frac{1-q^{2a}}{1-q^2}\right)\incg{spinors.32} \\
&=\left(\frac{1}{z q^{-a}+z^{-1}q^{a}}\right)
  \left(\frac{q^a-q^{-a}}{q-q^{-1}}\right)\incg{spinors.32} 
\end{align*}
\end{proof}

\begin{cor} For all $a\ge 1$
 \begin{equation*}
 \incg{spinors.33}=\left(
\prod_{k=1}^a \frac{1}{\{k\}}
\right) [a]! \dim_q(a) \incg{spinors.16} 
\end{equation*}
\end{cor}

\begin{prop}\label{abc} Let $(a,b,c)=(r+t,r+s,s+t)$ be admissible. Then
\begin{equation*}
 \incg{spinors.34}=\Delta\left(\prod_{k=1}^{r+s+t} \frac{1}{\{k\}}
\right)
\frac{[a]![b]![c]!}{[r]![s]![t]!}\incg{spinors.35} 
\end{equation*}
\end{prop}

\begin{proof} The proof is by induction on $a$. The basis of the induction is the case $a=0$
which is Proposition \ref{pid}. The inductive step is the following calculation.
 \begin{align*}
  \left(\frac{zq^{-m-n-p}+z^{-1}q^{m+n+p}}{q^{-m-n-p+1}}\right)&\incg{spinors.34}\\
=\sum_{r=0}^{c-1}{(-q)}^r&\incg{spinors.36}\\
&\qquad +\sum_{r=0}^{b-1}(-q)^{r+c}\incg{spinors.37}\\
=q^{c-1}[c]&\incg{spinors.38}
+q^{a+c-1}[b]\incg{spinors.39}
 \end{align*}

Now use the inductive hypothesis. This gives
\begin{multline*}
 \incg{spinors.34}=\Delta\left(\prod_{k=1}^{m+n+p} \frac{1}{\{k\}}
\right)
\frac{[a-1]![b]![c]!}{[m]![n]![p]!}\\
\left( q^{-m}[p]+q^p[m]\right)\incg{spinors.35}
\end{multline*}
\end{proof}

\subsection{Completeness}
Next we discuss completeness following \cite[(11.19)]{MR2418111}. This is based on the assumption
that we have a relation of the form
\begin{equation}\label{cp}
 \incg{spinors.121}=\sum_{m=0}^{\min(a,b)}C(m)\incg[.75in]{spinors.122}
\end{equation}
Then the problem is to determine the coefficients $C(m)$. This assumption is justified
by the representation theory.
\begin{prop} For $0\le 2m\le a+b$,
 \begin{equation}\label{cpl}
  C(m)=\left(\prod_{k=a+b-2m+1}^{a+b-m} \frac{1}{zq^{-k}+z^{-1}q^{k}}
\right) \qbinom{a}{m} \qbinom{b}{m} [m]! 
 \end{equation}
\end{prop}

\begin{proof}
The coefficients can be determined since this relation implies
\begin{equation*}
 \incg[1in]{spinors.123}=C(m)\incg[1in]{spinors.124}
\end{equation*}
Then the left hand side can be simplified using Proposition \ref{abc}
and the right hand side can be simplified using Proposition \ref{pid}
This gives an equation for $C(m)$.
\end{proof}

\begin{lemma}\label{lema} For $r\ge 0$,
\begin{equation*}
 \incg{spinors.99}=\frac{\{r\}}{\{r+1\}}\frac{[r+1]}{[r+2]}\incg{spinors.100}
\end{equation*}
\end{lemma}

\begin{proof} The proof is the following calculation.
\begin{equation*}
 \incg{spinors.99}= \incg{spinors.101}= 
\frac{\{r\}}{\{r+1\}}\frac{[r+1]}{[r+2]}
\incg{spinors.100}
 \end{equation*}
For the first step we use Definition \ref{RmatrixA}; the crossing is $S(q^{r+1})$. In the second step we use
crossing symmetry and Lemma \ref{lemS}.
\end{proof}

\begin{cor}\label{cor1} For $r\ge 0$,
 \begin{equation*}
 \incg[1in]{spinors.102}=
\frac{\{a\}}{\{a+r+1\}}\frac{[a+1]}{[a+r+2]}
\incg[1in]{spinors.103}
\end{equation*}
\end{cor}
\begin{proof}
 The proof is by induction on $r$. The basis of the induction is the case $r=0$
which is the case $r=0$ of Lemma \ref{lema}. The inductive step is an application
of Lemma \ref{lema}.
\end{proof}

\begin{prop}\label{tri} For all $a,b,m\ge 0$,
 \begin{equation*}
 \incg{spinors.115}=\prod_{k=0}^{m-1}
\frac{\{k\}}{\{a+k\}\{b+k\}}[2n-a-b-k]
\incg{spinors.116}
\end{equation*}
\end{prop}

\begin{proof} The proof is by induction on $m$. The basis of the induction is the case $m=0$.
The inductive step is the following calculation.
\begin{multline*}
 \incg[.75in]{spinors.117}=\incg[.75in]{spinors.119}-
\frac{[a+m]}{\{a+m\}}\incg[.75in]{spinors.120}\\
=[n]\incg[.75in]{spinors.115}-\frac{[b+m]}{\{b+m\}}\incg[1in]{spinors.115}\\
-\frac{\{a-1\}}{\{a+m-1\}\{a+m\}}[a]
\incg[.75in]{spinors.118} \\
=\frac{[2n-b-m]}{\{b+m\}}\incg[.75in]{spinors.115}\\
-\frac{\{a-1\}}{\{a+m-1\}\{a+m\}}[a]
\incg[.75in]{spinors.118}
\end{multline*}
The first step is an application of Proposition \ref{gamma} with $p=a+m-1$. Then in the second step we use 
Proposition \ref{gamma} with $p=b+m-1$ for the first diagram and Corollary \ref{cor1} with $r=m-1$ for the second diagram.

Now apply the inductive hypothesis and the identity
\[ [2n-b-m]\{a+m\}-[a]\{b\}=[2n-a-b-m]\{m\} \]
\end{proof}

Now we can calculate the $3j$-symbols.
\begin{lemma}\label{theta}
\begin{equation*}
\incg{spinors.141} =
\Delta \left( \prod_{k=1}^a \frac{1}{\{k\}}\right) \frac{\{1\}}{\{0\}}\frac{[2n]!}{[2n-a]!}
\end{equation*}
\end{lemma}
\begin{proof}
This is an application of Proposition \ref{pid} and the formula for $\dim_q(c)$ in Proposition \ref{dimq}.
\end{proof}

\begin{defn}
 For $r,s,t\ge 0$, the coefficient $X(r,s,t)$ is given by
\begin{equation*} X(r,s,t)=
\frac
{%
 \left( \prod_{k=0}^{r-1}{\{k\}}\right) 
 \left( \prod_{k=0}^{s-1}{\{k\}}\right) 
 \left( \prod_{k=0}^{t-1}{\{k\}}\right) 
}
{%
  \left( \prod_{k=0}^{r+s-1}{\{k\}}\right) 
 \left( \prod_{k=0}^{r+t-1}{\{k\}}\right) 
 \left( \prod_{k=0}^{s+t-1}{\{k\}}\right) 
} 
\end{equation*}
\end{defn}

\begin{prop}\label{threej} Let $(a,b,c)=(r+t,r+s,s=t)$ be admissible.
\begin{equation*}
  \incg{spinors.140}=\Delta X(r,s,t)\frac{[2n]!}{[2n-r-s-t]!} 
\end{equation*}
\end{prop}
\begin{proof} This follows from the following application of Proposition \ref{tri}
and Lemma \ref{theta}.
\begin{equation*}
\incg{spinors.140} =
\prod_{k=0}^{r-1}
\frac{\{k\}}{\{s+k\}\{t+k\}}[2n-s-t-k]
\incg{spinors.141} 
\end{equation*}
\end{proof}

\begin{cor} For all $r,s,t\ge 0$, 
 \begin{equation*}
  \incg{spinors.40}=X(r,s,t)\left( \prod_{k=1}^{r+s+t}\{k\}\right)
\frac{[r]![s]![t]!}{[r+s]![r+t]![s+t]!}\frac{[2n]!}{[2n-r-s-t]!}
 \end{equation*}
\end{cor}

\begin{proof} This follows from the following application of Proposition \ref{abc}.
\begin{equation*}
\incg{spinors.40}=
\frac{1}{\Delta}\left( \prod_{k=1}^{r+s+t} \{k\}\right) 
\frac{[r]![s]![t]!}{[a]![b]![c]!}
\incg{spinors.140} 
\end{equation*}
Now apply Proposition \ref{threej}.
\end{proof}

\begin{cor} For all $r,s,t\ge 0$, 
 \begin{equation*}
  \incg{spinors.153}=\Delta^2\left( \prod_{k=0}^{r+s+t-1}\frac1{\{k\}}\right) 
X(r,s,t) \frac{[a]![b]![c]!}{[r]![s]![t]!}\frac{[2n]!}{[2n-r-s-t]!}
 \end{equation*}
\end{cor}

\begin{proof} This follows from the following application of Proposition \ref{abc}.
\begin{equation*}
\incg{spinors.153}=
\Delta \left( \prod_{k=0}^{r+s+t-1} \right) \frac{[a]![b]![c]!}{[r]![s]![t]!}
\incg{spinors.140} 
\end{equation*}
\end{proof}

\begin{defn} The tetrahedron symbols are the following labelled strand networks.
\[ \incg[1in]{spinors.154} \]
\end{defn}
The label $i,j$ on a strand will be denoted $z_{ij}$. Then for $1\le i\le 3$ and $1\le j\le 4$
we put
\[ r_i = \sum_{j=1}^4 z_{ij} \qquad c_j = \sum_{i=1}^3 z_{ij} \]
Then we have $\sum_{i=1}^3 r_i=\sum_{j=1}^4c_j$.

\subsection{Fierz coefficients}
In this section we discuss the $q$-analogue of the Fierz coefficients.
The coefficients $F(a,b)$ are defined
by any of the following equivalent definitions. These satisfy $F(a,b)=F(b,a)$.

\begin{defn}
\begin{align}
 \incg{spinors.128}&=\Delta F(a,b) \\
\incg{spinors.129}&=\phi(b)F(a,b)\incg{spinors.130} \\
\incg{spinors.131}&=\sum_b \frac{F(a,b)}{\dim_q(b)}\incg{spinors.132}
\end{align}
\end{defn}

These coefficients can be determined using the method in \cite[\S 11.3]{MR2418111}.
The starting point is the observation:
\begin{equation*}
 \incg{spinors.128}=\incg{spinors.133}
\end{equation*}
Then we apply completeness to obtain:
\begin{equation*}
 \incg{spinors.133}=\sum_{m=0}^{\min (a,b)}C(m)\incg[1in]{spinors.134}
\end{equation*}
where from \eqref{cpl}, $C(m)$ is given by
\begin{equation*}
  C(m)=\left(\prod_{k=a+b-2m+1}^{a+b-m} \frac{1}{\{k\}}
\right) \qbinom{a}{m} \qbinom{b}{m} [m]!
 \end{equation*}
Then applying Lemma \ref{twist} gives
\begin{equation*}
 \incg[.9in]{spinors.134}=(-1)^{ab-m^2}z^{-2m}q^{ab-2(a-m)(b-m)}
\incg[.9in]{spinors.142}
\end{equation*}
Then from the proof of Proposition \ref{threej} we have
\begin{equation*}
\incg[1in]{spinors.142}=\Delta X(a-m,b-m,m) \frac{[2n]!}{[2n-a-b+m]!}
\end{equation*}

This gives an expression for $F(a,b)$.
It is not clear from this expression that these satisfy $F(a,b)=F(b,a)$
or that $F(a,b)$ is invariant under the involution $q,z\leftrightarrow q^{-1},z^{-1}$.

The coefficients $F(a,0)$ for $a\ge 0$ are determined by Proposition \ref{pid}. This gives
\begin{equation}
 F(a,0)=\left( \prod_{k=0}^{a-1}\frac{1}{\{n-k\}}\right) 
\end{equation} 

Next we calculate the coefficients $F(a,1)$ for $a\ge 1$ explicitly.
\begin{lemma}
\begin{equation*}
F(a,1)=(-1)^a\left( \prod_{k=0}^{a}\frac1{\{k\}}\right) [2n-2a]\frac{[2n]!}{[2n-a]!} 
\end{equation*}
\end{lemma}
\begin{proof}
The first step is
\begin{equation*}
\incg[.75in]{spinors.147}=
\incg[.75in]{spinors.148}
+\frac{1}{\{a+1\}}[a+1]
\incg[.75in]{spinors.149}
\end{equation*}
Then we have 
\begin{equation*}
\incg[1in]{spinors.148}=\Delta (-1)^{a+1}q^{-a-1}
\left( \prod_{k=0}^{a+1}\frac1{\{k\}}\right) \frac{[2n]!}{[2n-a-2]!}
\end{equation*}
\begin{equation*}
\incg[1in]{spinors.149}=\Delta (-1)^{a}z^{-2}q^{a+1}\left( \prod_{k=0}^{a}\frac1{\{k\}}\right)\frac{[2n]!}{[2n-a-1]!}
\end{equation*}
Hence
\begin{align*}
 F(a+1,1)&=(-1)^{a+1}q^{-a-1}\left( \prod_{k=0}^{a+1}\frac1{\{k\}}\right) \frac{[2n]!}{[2n-a-2]!} \\
&\qquad +\frac{1}{\{a+1\}}[a+1](-1)^{a}z^{-2}q^{a+1}\left( \prod_{k=0}^{a}\frac1{\{k\}}\right) \frac{[2n]!}{[2n-a-1]!} \\
&=(-1)^{a+1}\left( \prod_{k=0}^{a+1}\frac1{\{k\}}\right) \frac{[2n]!}{[2n-a-1]!} \\
&\qquad \left( q^{-a-1}[2n-a-1]-z^{-2}q^{a+1}[a+1]\right)
\end{align*}
\end{proof}

Next we find a recurrence relation for these coefficients.
\begin{lemma} For all $a,b,c\ge 0$,
\begin{equation*}
 \incg[.5in]{spinors.143}=\Delta \left( \prod_{k=0}^{c-1}\frac{\{k\}}{[2n-k]}\right) F(a,c)F(b,c)
\end{equation*}
\end{lemma}

\begin{proof} This is a straightforward calculation from the definition of $F(a,b)$.
\end{proof}

Then by Proposition
\begin{equation*}
 \incg[.5in]{spinors.144}=\incg[.5in]{spinors.145}
-\frac{[a+1][2n-a]}{\{a+1\}\{a\}}
\incg[.5in]{spinors.146}
\end{equation*}
This then gives the recurrence relation
\begin{equation*}
 F(a+2,b)=\frac{[2n-b]}{\{b\}}F(a+1,b)
-\frac{[a+1][2n-a]}{\{a+1\}\{a\}}F(a,b)
\end{equation*}

\def\cprime{$'$}

\end{document}